\documentclass[11pt, reqno, english]{amsart}

\usepackage[utf8]{inputenc}
\usepackage[T1]{fontenc}
\usepackage{amsmath,amsthm}
\usepackage{amsfonts,amssymb}
\usepackage{url}
\usepackage{mathtools}  
\usepackage[colorlinks=true,urlcolor=blue,linkcolor=blue,citecolor=blue]{hyperref}
\usepackage{enumerate, paralist}
\usepackage{tikz,tikz-cd}
\usetikzlibrary{calc,arrows}
\usetikzlibrary{math} % Using variables inside tikz
\usetikzlibrary{arrows,decorations,backgrounds,calc,math}
\tikzstyle{map}=[->,semithick]
\tikzstyle{arc}=[bend left,->,semithick]
\tikzstyle{rinclusion}=[right hook->,semithick]
\tikzstyle{linclusion}=[left hook->,semithick]

\usepackage[left=1in,right=1in,top=1in,bottom=1in]{geometry}
\usepackage{xcolor}
\usepackage{hyperref}
\usepackage{float}
\usepackage{subfigure}
\usepackage{color}

\newtheoremstyle{myremark} % name
{7pt}                    % Space above
{7pt}                    % Space below
{}  	                 % Body font
{}                           % Indent amount
{\bf}       	         % Theorem head fuont
{.}                          % Punctuation after theorem head
{.5em}                       % Space after theorem head
{}  % Theorem head spec (can be left empty, meaning ‘normal’)

\theoremstyle{plain}
\newtheorem{lemma}{Lemma}[section]
\newtheorem{theorem}[lemma]{Theorem}
\newtheorem{corollary}[lemma]{Corollary}
\newtheorem{proposition}[lemma]{Proposition}

\theoremstyle{definition}

\newtheorem*{definition-cnk}{Definition~\ref{def:cnk}}

\newtheorem*{question-motivating}{Motivating Question}


\newcounter{parentnumber}

\theoremstyle{myremark}

\newcommand{\R}{\mathbb{R}}

\newcommand{\cU}{\mathcal{U}}
\newcommand{\cV}{\mathcal{V}}
\newcommand{\cW}{\mathcal{W}}

\newcommand{\diam}{\mathrm{diam}}

\newcommand{\id}{\mathrm{id}}

\newcommand{\gh}{\mathrm{GH}}

\newcommand{\supp}{\mathrm{supp}}

\newcommand{\ord}{\mathrm{ord}}
\newcommand{\card}{\mathrm{card}}

\newcommand{\vr}[2]{\mathrm{VR}(#1;#2)}
\newcommand{\viet}[1]{\mathrm{V}(#1)}
\newcommand{\vietm}[1]{\mathrm{V}^\mathrm{m}(#1)}

\newcommand{\vrleq}[2]{\mathrm{VR}_\le(#1;#2)}
\newcommand{\vrm}[2]{\mathrm{VR}^\mathrm{m}(#1;#2)}

\newcommand{\vrmleq}[2]{\mathrm{VR}^\mathrm{m}_\le(#1;#2)}
\newcommand{\cech}[2]{\mathrm{\check{C}}(#1;#2)}
\newcommand{\acech}[3]{\mathrm{\check{C}}(#1,#2;#3)}
\newcommand{\cechm}[2]{\mathrm{\check{C}}^\mathrm{m}(#1;#2)}
\newcommand{\acechm}[3]{\mathrm{\check{C}}^\mathrm{m}(#1,#2;#3)}

\usepackage{accents}
\usepackage[normalem]{ulem}
\usepackage{wasysym}
\usepackage{charter}
\usepackage{euler}
\usepackage[T1]{fontenc}
\usepackage{thmtools}
\usepackage{thm-restate}

\begin{document}

\title{Vietoris thickenings and complexes of manifolds are homotopy equivalent}

\author{Henry Adams}
\author{Alexandre Karassev}
\author{\v{Z}iga Virk}

\begin{abstract}
We show that if $X$ is a finite-dimensional Polish metric space, then the natural bijection $\vr{X}{r}\to\vrm{X}{r}$ from the (open) Vietoris--Rips complex to the Vietoris--Rips metric thickening is a homotopy equivalence.
This occurs, for example, if $X$ is a Riemannian manifold.
The same is true for the map \v{C}$(X;r)\to$\v{C}$^\mathrm{m}(X;r)$
%$\cech{X}{r}\to\cechm{X}{r}$ \note{I don't know why this isn't compiling correctly} 
from the \v{C}ech complex to the \v{C}ech metric thickening, and more generally, for the natural bijection $\viet{\cW}\to\vietm{\cW}$ from the Vietoris complex to the Vietoris metric thickening of any uniformly bounded cover $\cW$ of a finite dimensional Polish metric space.
We also show that if $X$ is a compact metrizable space, then $\vietm{\cW}$ is strongly locally contractible.
\end{abstract}

\subjclass[2020]{55N31, 54F45, 54C55, 54C65}

\keywords{Vietoris--Rips complexes, \v{C}ech complexes, metric thickenings, dimension theory, absolute neighborhood retracts, Michael's selection theorem}

\maketitle

%\setcounter{tocdepth}{2}
%\tableofcontents

\section{Introduction}

Let $X$ be a metric space, and let $r>0$.
The \emph{Vietoris--Rips simplicial complex $\vr{X}{r}$} has $X$ as its vertex set, and a finite subset $\sigma\subseteq X$ as a simplex when its diameter is less than $r$.
Vietoris--Rips complexes were invented by %Leopold
Vietoris in a cohomology theory for metric spaces~\cite{Vietoris27,lefschetz1942algebraic}.
Independently, they were introduced by %Eliyahu
Rips in geometric group theory as a natural way to thicken a space.
Indeed, Rips used these complexes to show that torsion-free hyperbolic groups have Eilenberg--MacLane spaces that are finite simplicial complexes~\cite{bridson2011metric}.

More recently, Vietoris--Rips complexes and persistent homology have become commonly used tools in applied and computational topology, motivated by applications to
data analysis~\cite{CarlssonIshkhanovDeSilvaZomorodian2008,ghrist2008barcodes}.
%Since we do not know \emph{a priori} how to choose the scale $r$, the idea of persistent homology is to compute the homology of the Vietoris--Rips complex of a dataset $X$ over a large range of scale parameters $r$ and to trust those topological features which persist.
For $M$ a manifold and scale $r$ sufficiently small depending on the curvature of $M$, Hausmann~\cite{Hausmann1995} proves there is a homotopy equivalence $\vr{M}{r} \simeq M$.
Furthermore, for dataset $X$ sufficiently close to $M$ in the Gromov--Hausdorff distance, Latschev~\cite{Latschev2001} proves there is a homotopy equivalence $\vr{X}{r} \simeq M$.

If $X$ is not a discrete metric space, then the inclusion $X\hookrightarrow\vr{X}{r}$ is not continuous, since the vertex set of a simplicial complex is equipped with the discrete topology.
For example, if $M$ is a manifold of dimension at least one, then $M\hookrightarrow\vr{M}{r}$ is not continuous.
This observation helped motivate the definition of Vietoris--Rips metric thickenings in~\cite{AAF}.
For $X$ a metric space and $r>0$, the \emph{Vietoris--Rips metric thickening $\vrm{X}{r}$} is the space of finitely-supported probability measures on $X$ whose support has diameter less than $r$, equipped with an optimal transport metric.
The inclusion $X\hookrightarrow\vrm{X}{r}$, obtained by mapping a point $x\in X$ to the Dirac delta measure $\delta_x$, is now an isometric embedding.

What is the relationship between Vietoris--Rips simplicial complexes and metric thickenings?
For $X$ a finite metric space, $\vr{X}{r}$ and $\vrm{X}{r}$ are homeomorphic~\cite%[Proposition~6.2]
{AAF}.
For $X$ an arbitrary metric space, Gillespie showed that the natural bijection $\vr{X}{r}\to\vrm{X}{r}$ is a weak homotopy equivalence, i.e., induces isomorphisms on all homotopy groups~\cite{gillespie2024vietoris}, improving upon~\cite{HA-FF-ZV}.
The question remains if $\vr{X}{r}$ and $\vrm{X}{r}$ are homotopy equivalent.
Since $\vr{X}{r}$ is a simplicial complex (and hence a CW complex), this result would be implied by Whitehead's theorem if $\vrm{X}{r}$ were known to be an absolute neighborhood retract, %or an absolute neighborhood extensor, 
but to the best of our knowledge this question remains unknown for general metric spaces $X$.

We prove that if $X$ is a finite-dimensional Polish metric space and $r>0$, then
\begin{itemize}
\item $\vrm{X}{r}$ is an absolute neighborhood retract, and hence
\item the natural bijection $\vr{X}{r}\to\vrm{X}{r}$ from the (open) Vietoris--Rips complex to the Vietoris--Rips metric thickening is a homotopy equivalence.
\end{itemize}
This occurs, for example, if $X=M$ is a Riemannian manifold.

As a result, we get a version of Latschev's theorem for metric thickenings, resolving~\cite[Conjecture~6.9]{AAR} in the affirmative; see Corollary~\ref{cor:latschev}.

More generally, for $\cW$ a uniformly bounded cover of a Polish metric space, we show that the natural bijection $\viet{\cW}\to\vietm{\cW}$ from the Vietoris complex to the Vietoris metric thickening is a homotopy equivalence; see Theorem~\ref{thm:main}.
This also implies that the map $\cech{X}{r}\to\cechm{X}{r}$ from the (open) \v{C}ech complex to the \v{C}ech metric thickening is a homotopy equivalence.

Furthermore, if $X$ is a compact metrizable space, then we show that $\vietm{\cW}$ is strongly locally contractible, using Michael's selection theorem~\cite{michael1956continuous}.

We begin in Section~\ref{sec:preliminaries} with background and definitions.
In Section~\ref{sec:homotopy-equiv} we prove our main result, Theorem~\ref{thm:main}.
We prove that $\vietm{\cW}$ is strongly locally contractible for $X$ a compact metrizable space in Section~\ref{sec:strong-local-contractibility}.
In Section~\ref{sec:conclusion} we conclude with a list of open questions.

%\section{Related work}

\section{Preliminaries}
\label{sec:preliminaries}

We provide some background on metric spaces, Vietoris--Rips and \v{C}ech simplicial complexes, Vietoris simplicial complexes,  spaces of measures, metric thickenings, absolute neighborhood retracts, Whitehead's theorem, dimension, and local contractibility.

\subsection{Metric spaces}

Let $X=(X,d)$ be a metric space.
The \emph{diameter} of a subset $A\subseteq X$ is defined as $\diam(A)=\sup\{d(a,a')~|~a,a'\in A\}$.
The \emph{open ball} of radius $r>0$ centered at $x\in X$ is $B(x;r)=\{y\in X~|~d(y,x)<r\}$, which we sometimes denote as $B_X(x;r)$.

\subsection{Vietoris--Rips and \v{C}ech simplicial complexes}

Let $X$ be a metric space and let $r>0$.

The \emph{Vietoris--Rips simplicial complex $\vr{X}{r}$} has $X$ as its vertex set, and a finite subset $\sigma\subseteq X$ as a simplex when $\diam(\sigma)<r$.

The \emph{intrinsic \v{C}ech simplicial complex} $\cech{X}{r}$ has $X$ as its vertex set, and a finite subset $\sigma\subseteq X$ as a simplex when $\cap_{x\in \sigma}B(x;r)\neq\emptyset$.

If $Z\supseteq X$ is a metric space extending the metric on $X$, then the \emph{ambient \v{C}ech simplicial complex} $\acech{X}{Z}{r}$ has $X$ as its vertex set, and a finite subset $\sigma\subseteq X$ as a simplex when $\cap_{x\in \sigma}B_Z(x;r)\neq\emptyset$.

A common example of an ambient \v{C}ech complex is when $X\subseteq Z=\R^n$.
When $Z=X$, we note that the ambient \v{C}ech complex recovers the intrinsic \v{C}ech complex, namely $\acech{X}{X}{r}=\cech{X}{r}$.
For this reason, we can state our results with ambient \v{C}ech complexes while still handling the case of intrinsic \v{C}ech complexes.

We identify simplicial complexes with their geometric realizations.

\subsection{Vietoris simplicial complexes}
\label{ssec:vietoris}

Given a cover $\cW$ of a metric space, the \emph{Vietoris complex} $\viet{\cW}$ of $\cW$ has $X$ as its vertex set, and a finite subset $\sigma \subseteq X$ as a simplex when $\sigma$ is contained in an element of $\cW$.
By Dowker duality~\cite{dowker1952topology}, $\viet{\cW}$ is homotopy equivalent to the nerve of $\cW$.

Let $ r>0$.
When $\cW$ consists of all open subsets of diameter less than $r$, we have $\viet{\cW}=\vr{X}{r}$.
If instead $\cW=\{B_X(x;r)\}_{x\in X}$ consists of all open $r$-balls in $X$, then $\viet{\cW}=\cech{X}{r}$.
For $Z\supseteq X$ a metric space extending the metric on $X$, if $\cW=\{X\cap B_Z(z;r)\}_{z\in Z}$, then $\viet{\cW}=\acech{X}{Z}{r}$.

\subsection{Spaces of measures}

We follow~\cite{villani2008optimal,fedorchuk1991probability}.

%\note{Define a probability measure.}

For $X$ a topological space, we let $P(X)$ be the space of probability measures on the Borel sigma algebra of $X$, equipped with the weak topology.
The weak topology on $P(X)$ is induced by convergence against $C_b(X)$, the bounded continuous test functions.
In more detail, a sequence $\mu_1, \mu_2, \mu_3, \ldots$ in $P(X)$ is said to converge weakly to $\mu\in P(X)$ if for all bounded continuous functions $f\colon X\to \R$, we have $\lim_{n \to \infty} \int_X f(x)\,d\mu_n = \int_X f(x)\,d\mu$.

For $k\ge 0$, let $P_k(X)$ be the set of all measures in $P(X)$ whose support have cardinality at most $k$.
So, a measure $\mu \in P_k(X)$ can be written as $\mu=\sum_{i=0}^{k-1}\lambda_i \delta_{x_i}$ for some $x_i\in X$ and $\lambda_i \ge 0$ with $\sum_{i=0}^{k-1}\lambda_i=1$.
The \emph{support} of a measure $\mu = \sum_{i=0}^{k-1}\lambda_i \delta_{x_i} \in P_k(X)$ is $\supp(\mu)=\{x_i~|~\lambda_i>0\}$.

If $(X,d)$ is a metric space, then we may also define optimal transport metrics on $\cup_{k\ge \infty}P_k(X)$, the space of all finitely supported measures.
Given finitely supported measures $\mu_1$ and $\mu_2$, a \emph{coupling} between them is a probability measure $\mu$ on $X\times X$ with marginals $\mu_1$ and $\mu_2$.
This means $(\pi_1)_\sharp\mu=\mu_1$ and $(\pi_2)_\sharp \mu= \mu_2$, where $\pi_i\colon X\times X\to X$ is the projection map defined by $\pi_i(x_1,x_2)=x_i$ for $i=1,2$, and where $\sharp$ denotes the pushforward.
For $q\in[1,\infty)$ we may equip $\cup_{k\ge \infty}P_k(X)$ with the $q$-\emph{Wasserstein-Kantorovich-Rubinstein distance}, given by
\begin{equation}
d_{W^q}(\mu_1,\mu_2) = \inf_{\mu} \left(\int_{X\times X}(d(x,x'))^q\,d\mu\right)^{\frac{1}{q}},
\end{equation}
where $\mu$ varies over all couplings between $\mu_1$ and $\mu_2$.

%It is known that on a bounded metric space, for any $q \in [1, \infty)$, the $q$-Wasserstein metric generates the weak topology.
%For a summary of the result, see Appendix~\ref{app:Metrization_weak_topology}.
%On the other hand, the $\infty$-Wasserstein metric $\dWqq{\infty}$ generates a finer topology than the weak topology in general.

A \emph{Polish space} $X$ is a separable completely metrizable topological space, that is, a space homeomorphic to a complete metric space that has a countable dense subset.

%\note{For $q\in [1,\infty)$, define the $q$-Wasserstein-Kantarovoich-Rubenstein metric on $P(X)$.}

Let $k\ge 1$, and let $q\in [1,\infty)$.
If $X$ is a bounded Polish metric space, then the $q$-Wasserstein-Kantarovoich-Rubenstein metric on $P_k(X)$ agrees with the topology induced by weak convergence on $P_k(X)$ by~\cite[Corollary~6.13]{villani2008optimal}.
In turn, the topology induced by weak convergence coincides with the weak topology on $P(X)$ (see, e.g.,~\cite[Definitions~8.1.1 and 8.1.2]{bogachev2007measure} and~\cite{kallianpur1961topology}).
As a consequence, different choices of $q\in [1,\infty)$ all induce the same topology on $P_k(X)$, and we have the following observation.

\begin{proposition}
\label{proposition:weaktopology}
Let $X$ be a bounded Polish space.
Then the weak topology on $P_k(X)$ coincides with the topology induced by the $q$-Wasserstein-Kantarovoich-Rubenstein metric for any $q\in [1,\infty)$.
\end{proposition}

\subsection{Metric thickenings}

For $X$ a metric space and $r>0$, the \emph{Vietoris--Rips metric thickening} is the space
\[\vrm{X}{r}=\left\{\sum_{i=0}^{k-1}\lambda_i\delta_{x_i}\in P(X)~\mid~k\ge 1,\ \lambda_i\ge 0,\ \sum_i\lambda_i=1,\ \diam(\{x_0,\ldots,x_{k-1}\})<r\right\},\]
equipped with the optimal transport (or $1$-Wasserstein-Kantarovoich-Rubenstein) metric.

More generally, for an open cover $\cW$ of $X$, the \emph{metric thickening} of $\viet{\cW}$ is the space
\[
\vietm{\cW}=\left\{\sum_{i=0}^{k-1}\lambda_i\delta_{x_i}\in P(X)~\mid~k\ge 1,\ \lambda_i\ge 0,\ \sum_i\lambda_i=1,\ \exists W \in \cW: \{x_0,\ldots,x_{k-1}\} \subseteq W \right\},
\]
equipped with the optimal transport metric.
Choosing $\cW$ as in Section~\ref{ssec:vietoris} gives the \emph{\v{C}ech metric thickenings} $\cechm{X}{r}$ and $\acechm{X}{Z}{r}$.

The natural bijection $\vr{X}{r}\to \vrm{X}{r}$ is the continuous map defined by sending a point $\sum \lambda_i x_i$ in the geometric realization of $\vr{X}{r}$ to the probability measure $\sum \lambda_i\delta_{x_i}$ in $\vrm{X}{r}$.
The same is true for the natural bijection $\viet{\cW}\to \vietm{\cW}$.

In Corollary~\ref{cor:main} we show that for $X$ a finite-dimensional Polish metric space, the natural bijection $\vr{X}{r}\to\vrm{X}{r}$ is a homotopy equivalence.\footnote{
We remark this result is false if we had defined the Vietoris--Rips complex and metric thickening with the $\le$ convention instead of the $<$ convention, meaning a simplex is included in $\vr{X}{r}$ if its diameter is $\le r$ instead of $<r$.
Indeed, $\vrleq{X}{0}$ is $X$ with the discrete topology, whereas $\vrmleq{X}{0}$ is $X$ with its standard topology induced from its metric.
}

\subsection{Topological spaces}

For $Y$ a topological space and $y\in Y$, we let $\pi_n(Y,y)$ denote the $n$-th homotopy group of $Y$, with basepoint $y$.

A subset $A$ of a space $Y$ is a \emph{retract of $Y$} if there is a map $r\colon Y\to A$ with $r|_A=\id_A$.
This implies that $A$ is closed in $Y$.
A \emph{neighborhood retract of $Y$} is a closed set in $Y$
that is a retract of some neighborhood in $Y$.
A metrizable space $Y$ is called an
\emph{absolute neighborhood retract (ANR)} if $Y$ is a
neighborhood retract of any metrizable space that contains $Y$ as a closed subspace.
See~\cite{sakai2013geometric} for more information on ANRs.

By~\cite[Corollary~6.6.5]{sakai2013geometric} a space has the homotopy type of a simplicial complex if and only if it has the homotopy type of an ANR.

\subsection{Whitehead's theorem}

Let $Y$ and $Y'$ be topological spaces.
A map $f\colon Y\to Y'$ is a \emph{weak homotopy equivalence} if it induces
\begin{itemize}
\item a bijection from the path components of $Y$ to the path components of $Y'$, and
\item an isomorphism $\pi_n(Y,y)\to\pi_n(Y',y)$ for all $n\ge 1$ and $y\in Y$.
\end{itemize}

The classical version of Whitehead's theorem states that if $Y$ and $Y'$ are CW complexes, then every weak homotopy equivalence $f\colon Y\to Y'$ is a homotopy equivalence~\cite{Hatcher}.
%Simplicial complexes, including Vietoris--Rips simplicial complexes, are CW complexes.

The CW complex assumption can be relaxed, as follows.
By~\cite[Corollary~6.6.6]{sakai2013geometric}, if $Y$ and $Y'$ are ANRs, then every weak homotopy equivalence $f\colon Y\to Y'$ is a homotopy equivalence.

%We can combine Corollaries~6.6.5 and~6.6.6 of~\cite{sakai2013geometric} to get that if $K$ is a simplicial complex and $Y'$ is an ANR, then every weak homotopy equivalence $f\colon K\to Y'$ is a homotopy equivalence.
%Indeed, let $g\colon Y\to K$ be a homotopy equivalence from an ANR $Y$ to the simplicial complex $K$, which exists by~\cite[Corollaries~6.6.5]{sakai2013geometric}.
%Then $fg\colon Y\to Y'$ is a weak homotopy equivalence between ANRs, which is therefore a homotopy equivalence by~\cite[Corollaries~6.6.6]{sakai2013geometric}.
%Since $g$ and $fg$ are each homotopy equivalences, the two-out-of-three property~\note{\cite{}} then gives that $f$ is a homotopy equivalence, as desired.

\subsection{Dimension}

We refer the reader to Chapter~5 of Sakai's book~\cite{sakai2013geometric} for background on dimension.

Let $Y$ be a topological space, and let $\cU$ be an open cover of $Y$.
For any point $y\in Y$, let $\cU[y]=\{U\in\cU~|~y\in U\}$ be the subcollection of sets in $\cU$ containing $y$, and let $\card(\cU[y])$ be the cardinality of this set.
The \emph{order} of $\cU$ is $\ord(\cU)=\sup\{\card(\cU[y])~|~y\in Y\}$.
We define the \emph{(covering) dimension $\dim Y$ of $Y$} as follows.
Let $n$ be a nonnegative integer.
If every finite open cover of $Y$ has a finite open refinement $\cU$ with $\ord(\cU)\le n+1$, then $\dim Y\le n$.
(When $Y$ is paracompact, the word ``finite'' can be removed in both places in the prior sentence, still yielding an equivalent definition.)
We have $\dim Y=n$ if $\dim Y\le n$ and $\dim Y\nleq n-1$, in which case we say that the space $Y$ is \emph{$n$-dimensional}.
The space $Y$ is \emph{finite-dimensional} if $\dim Y\le n$ for some finite $n$, and otherwise $Y$ is \emph{infinite-dimensional}.

The space $Y$ is \emph{countable-dimensional} if it is a countable union of finite-dimensional normal subspaces.
%\note{Page 279 of Sakai writes:} It is said that $X$ is countable-dimensional (c.d.) if $X$ is a countable union of f.d.\ normal subspaces, where it should be noted that subspaces of normal spaces need not be normal (cf.\ Sect.\ 2.10).
Though a subspace of a normal space need not be normal, a subspace of a metric space is a metric space and hence is normal.
Therefore, a metric space $X$ is countable-dimensional if it is a countable union of finite-dimensional subspaces.

The countable sum theorem~\cite[Theorem~5.4.1]{sakai2013geometric} says that if the countable union $Y=\cup_{i\ge 1}F_i$ is normal, if each $F_i$ is closed in $Y$, and if $\dim F_i\le n$ for every $i\ge 1$, then $\dim Y\le n$.

Let $X$ be a metric space.
The subset theorem~\cite[Theorem~5.3.3]{sakai2013geometric} says that for every subset $A$ of a metrizable space $X$, we have $\dim A \le \dim X$.

The decomposition theorem~\cite[Theorem~5.4.5]{sakai2013geometric} states that a metric space $X$ satisfies $\dim X\le n$ if and only if $X$ is covered by $n+1$ many $0$-dimensional subsets $X_1$, \ldots, $X_{n+1}$.
It follows that a metric space $X$ is countable-dimensional if and only if it is a countable union of 0-dimensional subspaces.

% Recall from~\cite[Section~5.6]{sakai2013geometric} that $Y$ is \emph{countable-dimensional} if $Y=\cup_{i\ge 1} A_i$ for some countably many subsets $A_i\subseteq Y$ with $\dim A_i \le 0$.
%By~\cite[Theorem~5.4.5]{sakai2013geometric}, every $n$-dimensional metrizable space is the union of at most $n+1$ many $0$-dimensional subspaces.
%It follows that if $Y$ is the countable union of finite-dimensional subspaces, then $Y$ is countable-dimensional.

By~\cite[Theorem~6.10.1]{sakai2013geometric}, every countable-dimensional locally contractible metrizable space is an ANR.

\subsection{Local contractibility}
\label{ssec:local-contractibility}

A topological space $Y$ is \emph{locally contractible at $y \in Y$} if for every neighborhood $U$ of $y$, there is a neighborhood $y\in V\subseteq U$ such that the inclusion $V\hookrightarrow U$ is nullhomotopic in $U$.
The space $Y$ is \emph{locally contractible} if it is locally contractible at each $y\in Y$.

A topological space $Y$ is \emph{strongly locally contractible at $y \in Y$} if for every neighborhood $U$ of $y$, there is a contractible neighborhood $y\in V\subseteq U$.
In other words, if every point $y\in Y$ has a local base of contractible neighborhoods, then we say that $Y$ is \emph{strongly locally contractible}.

A neighborhood $U$ \emph{strongly} deformation retracts onto $y$ if the deformation retract fixes the point $y$ throughout the homotopy.
If every point $y\in Y$ has a local base of neighborhoods which strongly deformation retract onto $y$, then we say that $Y$ is \emph{strictly strongly locally contractible}.

\section{$\viet{\cW}\to\vietm{\cW}$ is a homotopy equivalence for $X$ finite-dimensional}
\label{sec:homotopy-equiv}

%Alex: If $\vrm{X}{r}$ is weakly infinite-dimensional (which should be implied if $X$ is finite-dimensional, such as a manifold), then strong local contractibility of $\vrm{X}{r}$ would give that $\vrm{X}{r}$ is an ANR.
%Hence $\vr{X}{r}\to\vrm{X}{r}$ would be a homotopy equivalence, (by Whitehead's theorem, since it is already known to be a weak homotopy equivalence).

%Theorem~0.42 on page 26 of (\footnote{\url{https://repovs.splet.arnes.si/files/2022/09/Continuous-Selections-of-Multivalued-Mappings-1998.pdf}}) works if $X$ is finite-dimensional, but we want a stronger version where $X$ is only assumed to be weakly infinite-dimensional.

%Metric space $X$ finite-dimensional, space of measures with at most $m$ points in support is finite-dimensional, so $\vrm{X}{r}$ is countably-dimensional.
%(Perhaps even strongly countably-dimensional, instead of weakly so.)

%Book: Katsuro Sakai, ``Geometric aspects of general topology'', page 392, theorem 6.10.1.

%Alex and Ziga are pointing out for $X$ finite-dimensional, (standard, not strong) local contractibility  along with Sakai's result is enougth to give that $\vr{X}{r}\to\vrm{X}{r}$ is a homotopy equivalence --- and we already know that $\vrm{X}{r}$ is locally contractible by our prior paper?

We prove that if $\cW$ is a uniformly bounded open cover of a finite-dimensional metric space $X$, then the Vietoris complex of $\cW$ and its metric thickening are homotopy equivalent.
A cover $\cW$ of $X$ is \emph{uniformly bounded} if there exists some $D>0$ such that $\diam(W) < D$ for all $W \in \cW$.

\begin{theorem}
\label{thm:main}
Let $\cW$ be a uniformly bounded open cover of a finite-dimensional Polish metric space $X$.
The natural bijection $\viet{\cW}\to\vietm{\cW}$ from the  Vietoris complex to the Vietoris metric thickening is a homotopy equivalence.
\end{theorem}

The following corollary is immediate by choosing $\cW$ to be as in Section~\ref{ssec:vietoris}.

\begin{corollary}
\label{cor:main}
Let $X$ be a finite-dimensional Polish metric space, and let $r>0$.
The natural bijection $\vr{X}{r}\to\vrm{X}{r}$ is a homotopy equivalence.
For $Z\supseteq X$ a metric space extending the metric on $X$, the natural bijection $\acech{X}{Z}{r}\to\acechm{X}{Z}{r}$ is a homotopy equivalence.
\end{corollary}

The proof of Theorem~\ref{thm:main} will rely on showing that in this setting, $\vietm{\cW}$ is countable-dimensional and an ANR.
The proof that $\vietm{\cW}$ is countable-dimensional is closely related to the fact that if $X$ is a finite-dimensional metric space, then for any integer $k$, the space of all probability measures with support of size at most $k$ is finite-dimensional.
See for example Equation~(3.1) on Page~57 of~\cite{fedorchuk1991probability}, 
%\url{https://www.mathnet.ru/links/7c55e926712f5859b5b382f2bab84f51/rm4568_eng.pdf}
which follows from a more general result of Basmanov~\cite{basmanov1983covariant}.

\begin{lemma}
\label{lem:countable-dimensional}
If $\cW$ is a uniformly bounded open cover of a finite-dimensional Polish metric space $X$, then $\vietm{\cW}$ is countable-dimensional.
\end{lemma}

%(\footnote{\note{Perhaps if $X$ is finite-dimensional, then $\vrm{X}{r}$ is even strongly countable-dimensional, though we don't need that here.}})

\begin{proof}
Recall
\[
\vietm{\cW}=\left\{\sum_{i=0}^{k-1}\lambda_i\delta_{x_i}\in P(X)~\mid~k\ge 1,\ \lambda_i\ge 0,\ \sum_i\lambda_i=1,\ \exists W \in \cW: \{x_0,\ldots,x_{k-1}\} \subseteq W \right\}.
\]
For any integer $k\ge 0$, let
\[
\vietm{\cW}_k=\left\{\sum_{i=0}^{k-1}\lambda_i\delta_{x_i}\in P(X)~\mid~\lambda_i\ge 0,\ \sum_i\lambda_i=1,\ \exists W \in \cW: \{x_0,\ldots,x_{k-1}\} \subseteq W \right\}
\]
be the set of all measures in $\vietm{\cW}$ with support of size at most $k$.
Note that we can write the Vietoris metric thickening as the countable union $\vietm{\cW}=\cup_{k\ge 1}\vietm{\cW}_k$.
So, in order to show that $\vietm{\cW}$ is countable-dimensional, it suffices to show that each $\vietm{\cW}_k$ is finite-dimensional.

Let $Y$ be a topological space.
For $k$ an integer, recall that $P_k(Y)$ is the set of measures with support of size at most $k$,
\[P_k(Y)=\{\mu\in P(Y)~|~\card(\supp(\mu))\le k\},\]
equipped with the weak topology.
%\note{Page~57 of~\cite{fedorchuk1991probability}, 
%%\url{https://www.mathnet.ru/links/7c55e926712f5859b5b382f2bab84f51/rm4568_eng.pdf}
%near Equation~(3.1), states the following.
%``One of the dimensional characteristics of a continuous and intersection preserving functor $\cF\subseteq P_n$ is the formula 
%\[\dim \cF(Y)\le n \dim Y + \dim \cF(n)\]
%that follows from a more general result of Basmanov~\cite{basmanov1983covariant}.
%This formula is valid for any (non-empty) compact space $Y$.''}
%Let $\comp$ be the category of compact metric spaces.
%\note{Question: Is $\cF$ considered as a functor $\cF\colon \comp\to\comp$?
%Do we hence need to add the assumption ``$X$ is compact'' to this lemma and to the subsequent lemma and theorem?}
%\note{Question: Does ``$\cF\subseteq P_n$'' mean $\cF(Y)\subseteq P_n(Y)$ for all $Y\in \comp$?}
%\note{Question: What does ``continuous'' mean here?}
%\note{Question: Does ``intersection preserving'' mean $\cF(\cap_\alpha S_\alpha)=\cap_\alpha \cF(S_\alpha)$?}
For $Y$ a completely regular space, Equation (5.1) of~\cite{fedorchuk1991probability} states $\dim(P_k(Y))\le k \dim(Y)+(k-1)$; see also Basmanov~\cite{basmanov1983covariant}.

Since $X$ is a metric space, it is completely regular, giving $\dim(P_k(X))\le k \dim(X)+(k-1) \eqqcolon D_k$.
Fix a point $a\in X$ and for any positive integer $i\ge 1$, let $X_n = \{x\in X\mid d(a,x)\le i \}$.
Then  the Wasserstein-Kantarovoich-Rubenstein metric on $P_k(X_i)$ agrees with the weak topology on $P_k(X_i)$ by Proposition~\ref{proposition:weaktopology}.
By the subset theorem~\cite[Theorem~5.3.3]{sakai2013geometric}, the inclusion $P_k(X_i)\subseteq P(X)$ gives $\dim(P_k(X_i))\le D_k$ for all integers $i \ge 1$.
Since $P_k(X) =\cup_{i\ge 1} P_k(X_i)$, the countable cum theorem~\cite[Theorem~5.4.1]{sakai2013geometric} implies $\dim P_k(X)\le D_k$.
Finally, the inclusion $\vietm{\cW}_k\subseteq P_k(X)$ gives $\dim(\vietm{\cW}_k)\le D_k$, again by the subset theorem.

%\note{Assuming we are allowed to take $\cF\coloneqq\vr{-}{r}_k\subseteq P_k$ [defined only on \emph{metric spaces}], then we would get
%\begin{align*}
%\dim \vr{X}{r}_k
%&\le k\dim X+\dim\cF(k) \\
%&= k\dim X + (k-1), \\
%\end{align*}
%which would complete the proof.}

Therefore $\vietm{\cW}=\cup_{k\ge 1}\vietm{\cW}_k$ is a countable union of finite-dimensional spaces, and hence is countable-dimensional.
\end{proof}

\begin{lemma}
\label{lem:anr}
If $\cW$ is a uniformly bounded open cover of a finite-dimensional Polish metric space $X$, then $\vietm{\cW}$ is an ANR.
\end{lemma}

\begin{proof}
The space $\vietm{\cW}$ is locally contractible by~\cite[Theorem~2]{HA-FF-ZV}, and it is countable-dimensional by Lemma~\ref{lem:countable-dimensional}.
Hence $\vietm{\cW}$ is an ANR by~\cite[Theorem~6.10.1]{sakai2013geometric}.
\end{proof}

We are now prepared to prove our main result.

\begin{proof}[Proof of Theorem~\ref{thm:main}]
Let $X$ be a finite-dimensional Polish metric space, and let $\cW$ be a uniformly bounded open cover of $X$.
Let $f\colon \viet{\cW}\to\vietm{\cW}$ denote the natural bijection.
We must show that $f$ is a homotopy equivalence

The natural bijection $f\colon\viet{\cW}\to\vietm{\cW}$ is a weak homotopy equivalence by Gillespie's work~\cite{gillespie2024vietoris}.
The space $\viet{\cW}$ is a simplicial complex by definition, and the space $\vietm{\cW}$ is an ANR by Lemma~\ref{lem:anr}.
%Therefore, the combination of Corollaries~6.6.5 and~6.6.6 of~\cite{sakai2013geometric} (as explained in Section~\ref{ssec:Whitehead} on Whitehead's theorem) give that this weak homotopy equivalence is furthermore a homotopy equivalence.
We can now combine Corollaries~6.6.5 and~6.6.6 of~\cite{sakai2013geometric} to get that $f\colon\viet{\cW}\to\vietm{\cW}$ is a homotopy equivalence.
Indeed, let $g\colon Y\to \viet{\cW}$ be a homotopy equivalence from a space $Y$ which is an ANR to the simplicial complex $\viet{\cW}$, which exists by~\cite[Corollary~6.6.5]{sakai2013geometric}.
Then $fg\colon Y\to \vietm{\cW}$ is a weak homotopy equivalence between ANRs, which is therefore a homotopy equivalence by~\cite[Corollary~6.6.6]{sakai2013geometric}.
Let $(fg)^{-1}$ denote a homotopy inverse of $fg$.
Since $g$ and $fg$ are each homotopy equivalences, the two-out-of-three property gives that the natural bijection $f\colon\viet{\cW}\to\vietm{\cW}$ is a homotopy equivalence (with $g(fg)^{-1}$ as a homotopy inverse), as desired.
\end{proof}

Theorem~\ref{thm:main} (or Corollary~\ref{cor:main}) has the following consequences.
The first is Latschev's theorem~\cite{Latschev2001} for metric thickenings.

\begin{corollary}
\label{cor:latschev}
Let $M$ be a closed Riemannian manifold.
There exists $r_0>0$ such that for every $0<r\le r_0$, there is some $\delta>0$ such that for any finite-dimensional Polish metric space $X$ with $d_\gh(X,M)<\delta$, we have $\vrm{X}{r}\simeq M$.
\end{corollary}

\begin{proof}
The result is true with Vietoris--Rips simplicial complexes by~\cite[Theorem~1.1]{Latschev2001}, and then we apply Corollary~\ref{cor:main}.
\end{proof}

The above result resolves \cite[Conjecture~6.9]{AAR} in the affirmative.

\begin{corollary}
\label{cor:interleaving}
Let $X$ be a finite-dimensional Polish metric space.

The natural bijection $\vr{X}{r}\to\vrm{X}{r}$ from the (open) Vietoris--Rips complex to the Vietoris--Rips metric thickening is a $0$-homotopy interleaving, and hence these filtrations have isomorphic persistent homology modules.

For $Z\supseteq X$ a metric space extending the metric on $X$, the natural bijection $\acech{X}{Z}{r}\to\acechm{X}{Z}{r}$ from the (open) \v{C}ech complex to the \v{C}ech metric thickening is a $0$-homotopy interleaving, and hence these filtrations have isomorphic persistent homology modules.
\end{corollary}

We remark that the existence of a $0$-homotopy interleaving~\cite[Definition~3.5]{blumberg2023universality} is a stronger than simply saying that the homotopy interleaving distance is zero --- indeed, it furthermores that the infimum in the definition of the homotopy interleaving distance~\cite[Definition~3.6]{blumberg2023universality} is attained.
When $X$ is a finite-dimensional Polish metric space, Corollary~\ref{cor:interleaving} implies~\cite[Corollary~3]{gillespie2024vietoris}.

\begin{proof}
This is a consequence of the fact that for all $r\le r'$ the following diagrams commute, where the horizontal arrows are inclusions, and where the vertical arrows are the natural bijections, which are homotopy equivalences.
\[
\begin{tikzcd}
	{\vr{X}{r}} & {\vr{X}{r'}} \\
	{\vrm{X}{r}} & {\vrm{X}{r'}}
	\arrow[hook, from=1-1, to=1-2]
	\arrow["\simeq"', from=1-1, to=2-1]
	\arrow["\simeq", from=1-2, to=2-2]
	\arrow[hook, from=2-1, to=2-2]
\end{tikzcd}
\hspace{15mm}
\begin{tikzcd}
	{\acech{X}{Z}{r}} & {\acech{X}{Z}{r'}} \\
	{\acechm{X}{Z}{r}} & {\acechm{X}{Z}{r'}}
	\arrow[hook, from=1-1, to=1-2]
	\arrow["\simeq"', from=1-1, to=2-1]
	\arrow["\simeq", from=1-2, to=2-2]
	\arrow[hook, from=2-1, to=2-2]
\end{tikzcd}
\]
\end{proof}

\section{Strong local contractibility of $\vietm{\cW}$}
\label{sec:strong-local-contractibility}

In this section we prove that for any uniformly bounded open cover $\cW$ of a compact metric space $X$, the space $\vietm{\cW}$ is strongly locally contractible --- in fact strictly strongly locally contractible.
This means, as defined in Subsection~\ref{ssec:local-contractibility}, that each $\mu\in \vietm{\cW}$ has a local base of neighborhoods which strongly deformation retract onto $\mu$.
As before, the cases when $\cW$ is chosen as in Section~\ref{ssec:vietoris} to recover $\vietm{\cW}=\vrm{X}{r}$ or $\vietm{\cW}=\acechm{X}{Z}{r}$ are especially pertinent.

Let $\cW$ be a uniformly bounded open cover of a compact metrizable space $X$.
By $P(X)$ we denote the set of all probability measures on $X$, and by $M(X)$ we denote the space of all measures of $X$.
By the Riesz theorem, $M(X)$ is isomorphic to the dual of $C(X)$, and so $M(X)$ is a Banach space.
In particular, it makes sense to discuss convex subsets of $P(X)$.

\subsection{Alternative basis}

Since $X$ is compact, we have already stated (see Proposition~\ref{proposition:weaktopology}) that the weak topology on $\vietm{\cW}$ coincides with the one induced by the Wasserstein metric.
In this subsection we will construct a convenient basis of this topology, consisting of sets $N(\eta,\cU,\varepsilon)$ defined below.

For a measure $\eta = \sum_{i=1}^n p_i \delta_{y_i}$ on $X$, a continuous functions $f\colon X\to \mathbb R$, and a subset $Y$ of $X$, we let
\[\eta(Y) = \sum_{y_i\in Y} p_i
\quad\mbox{ and }\quad
\eta(f) = \sum_{i=1}^n p_i f(y_i).\]

Recall that the basis of the weak topology on $P(X)$ is given by the sets of the form 
\[\widetilde{O}(\eta, h_1,\dots ,h_k, \varepsilon) =\{\nu \in P(X)~:~ |\eta(h_i)- \nu(h_i)|<\varepsilon \text{ for } i=1,2,\dots, k\},\]
where $h_1, h_2, \dots ,h_k$ are continuous real-valued functions on $X$ and $\varepsilon >0$.
We define
\[O(\eta, h_1,\dots ,h_k, \varepsilon) %= \{\nu \in \vrm{X}{r}~:~ |\eta(h_i)- \nu(h_i)|<\varepsilon \text{ for } i=1,2,\dots, k\}
= \widetilde{O}(\eta, h_1,\dots ,h_k, \varepsilon) \cap \vietm{\cW}.\]

Let $\eta = \sum_{i=1}^n p_i \delta_{y_i}$ be a measure in $\vietm{\cW}$.
Let $\cU=\{U_1, U_2, \dots, U_n\}$ be a collection of disjoint open subsets of $X$ with $y_i\in U_i$ for all $i$, and with $\cup\cU \coloneqq \cup_{i=1}^n U_i$ contained in an element of $\cW$.
For any $\varepsilon >0$, let
\[N(\eta,\cU,\varepsilon) =\left\{\nu \in \vietm{\cW}~:~|\eta(U_i)-\nu(U_i)|<\varepsilon \text{ for all }i\right\}.\]
We will see in Lemma~\ref{lemma:basis3} that the sets of the form $N(\eta,\cU,\varepsilon)$ form a basis for $\vietm{\cW}$ at $\eta$.

\begin{lemma}
\label{lemma:basis1} For each 
\[\nu  =\sum_{j=1}^m w_j \delta_{a_j}\in N(\eta,\cU,\varepsilon)\]
there exists a collection of disjoint open sets $\cV=\{V_1,\dots, V_m\}$ in $X$ and $\delta >0$ such that $\cup\cV$ is contained in an element of $\cW$, such that $a_j\in V_j$ for all $j=1,2,\dots,m$, and such that
\[N(\nu,\cV,\delta)\subseteq N(\eta,\cU,\varepsilon).\]
\end{lemma}

\begin{proof}
Let $\delta = \frac{1}{2m}\min\{\varepsilon - |\eta(U_i) - \nu(U_i)|:i=1,2,\dots,n\}$.
Since $\nu\in \vietm{\cW}$, we have that $\{a_1,\ldots,a_m\}$ is contained in an element of $\cW$.
Consider disjoint open neighborhoods $V_1,\dots , V_m$ of $a_1,\dots , a_m$ such that $V_j\subseteq U_i$ whenever $a_j\in U_i$, and such that $\cup\cV$ is contained in an element of $\cW$.
Let $\cV = \{V_1\ldots,V_m\}$.
Let $\zeta =\sum_{k=1}^l q_k \delta_{c_k}\in  N(\nu,\cV,\delta)$.
Then for all $i=1,2, \dots, n$, we have
\begin{align*}
|\nu(U_i) - \zeta (U_i)|
&\le |\nu(U_i) - \nu(U_i)| + |\nu(U_i) - \zeta(U_i)| \\
&< \varepsilon - 2m \delta + \sum_{\{j\,:\,V_j\subseteq U_i\}} |\nu (V_j) -\zeta (V_j)| + \left|\sum_{c_k\notin \cup\cV} q_k \right| \\
&< \varepsilon - 2m\delta + m\delta + m \delta = \varepsilon.
\end{align*}
Therefore $\zeta \in N(\nu,\cU,\varepsilon)$.
\end{proof}

%Let $\mu =\sum_{i=1}^n m_i \delta_{x_i}$ be a measure in $\vrm{X}{r}$.
%Let $\cU=\{U_1, U_2, \dots, U_n\}$ be a collection of disjoint open subsets of $X$ such that $x_i\in U_i$ for all $i$, and $\diam(\cU) <r$.

\begin{lemma}
\label{lemma:basis2}
For each set of the form $N(\eta,\cU,\varepsilon)$ there exist functions $h_1,h_2,\dots, h_k$ such that
\[O(\eta, h_1,\dots, h_k, \varepsilon)\subseteq N(\eta,\cU,\varepsilon).\]
\end{lemma}

\begin{proof}
Let $\eta = \sum_{i=1}^n p_i \delta_{y_i}$, and let $\cU=\{U_1, U_2, \dots, U_n\}$ be a collection of disjoint open subsets of $X$ with $y_i\in U_i$ for all $i$, and with $\cup\cU$ contained in an element of $\cW$.

Consider functions $f_1,\dots, f_n$ from $X$ to $[0,1]$ such that for all $i=1,2,\dots,n$, we have $f_i(U_i)=\{1\}$ and $f_i(U_j) =\{0\}$ for all $j\ne i$.
Also, consider functions $g_1,\dots, g_n$ from $X$ to $[0,1]$ such that for all $i=1,2,\dots,n$, we have $g_i(y_i) = 1$ and $g_i(X\setminus U_i) =\{0\}$.
We claim that 
\[O(\eta, f_1,\dots, f_n, g_1,\dots, g_n, \varepsilon)\subseteq N(\eta,\cU,\varepsilon).\]
Consider $\nu \in O(\eta, f_1,\dots, f_n, g_1,\dots, g_n, \varepsilon)$.
Note that 
\[\eta(U_i) - \nu (U_i) = \eta(f_i) - \nu (U_i) \le \eta (f_i) - \nu (f_i) <\varepsilon\]
and
\[\eta (U_i) - \nu(U_i) = \eta (g_i) - \nu (U_i) \ge \eta (g_i) - \nu (g_i) > -\varepsilon.\]
This implies $|\eta (U_i) - \nu (U_i)|< \varepsilon$ and therefore $\nu \in N(\eta,\cU,\varepsilon)$.
\end{proof}

\begin{lemma}
\label{lemma:basis3}
The sets of the form $N(\eta,\cU,\varepsilon)$ form a basis for $\vietm{\cW}$ at $\eta$.
\end{lemma}

\begin{proof}
Let $\eta = \sum_{i=1}^n p_i \delta_{y_i}$.
Lemmas~\ref{lemma:basis1} and~\ref{lemma:basis2} imply that sets of the form $N(\nu,\cU,\varepsilon)$ are open in the weak topology.
Consider an open set $O(\nu, h_1,\dots , h_k, \varepsilon)$.
Let
\[M = 1+\max \{ |h_t(x)|~:~ x\in X,\ t=1,2,\dots , k\},\]
and let $\delta = \frac{\varepsilon}{ M(2n+1)}$.
There exists a collection $\cU=\{U_1, \dots , U_n\}$ of disjoint open neighborhoods  of $y_1, \dots, y_n$ such that $\cup\cU$ is contained in an element of $\cW$, and such that for each $i=1,2,\dots, n$ we have $|h_t(y_i) - h_t(x)| < \delta$ for all $x\in U_i$ and $t=1,2,\dots, k$.
Pick $\nu =\sum_{j=1}^m w_j \delta _{a_j} \in N(\eta,\cU,\delta)$.
Then
\begin{align*}
|\eta (h_t) - \nu (h_t)|
&= \left|\sum_{i=1}^n p_i h_t(y_i) - \sum_{j=1}^m w_j h_t(a_j)\right| \\
&\le \left|\sum_{i=1}^n \left( p_i h_t(y_i) - \sum_{a_j\in U_i} w_j h_t(a_j)\right)\right| +\left|\sum_{a_j\notin \cup\cU} w_j h_t(a_j)\right| \\  
&\le \left|\sum_{i=1}^n\left(p_i h_t(y_i) - \sum_{a_j\in U_i} w_jh_t(y_i)\right)\right|+ \sum_{i=1}^n\left(\sum_{a_j\in U_i} w_j |h_t(y_i)-h_t(a_j)|\right)  +M\sum_{a_j\notin \cup\cU} w_j \\
&< M\sum_{i=1}^n \left|p_i - \sum_{a_j\in U_i} w_j\right|+\delta + Mn\delta \\
&\le M\sum_{i=1}^n|\eta(U_i) - \nu (U_i)|+M(n+1)\delta \\
&\le Mn\delta +M(n+1)\delta = \varepsilon.
\end{align*}
 Therefore $\nu\in O(\eta , h_1,\dots, h_k,\varepsilon)$.
\end{proof}

\subsection{Strong local contractibility via Michael's selection theorem for multivalued maps}
\label{ssec:slc}

Our proof that $\vietm{\cW}$ is strictly strongly locally contractible in Theorem~\ref{thm:strong-local-contractibility} will use Michael's selection theorem~\cite{michael1956continuous}, which we recall now.

A multivalued map $F\colon X\to Y$ assigns to each point $x\in X$ a subset $F(x)\subseteq Y$.
A \emph{selection} of $F$ is a function $f\colon X\to Y$ with $f(x)\in F(x)$ for all $x\in X$.

\begin{theorem}[Michael's selection theorem~\cite{michael1956continuous}]
\label{theorem:selection}
Let $X$ be a paracompact space, let $Y$ be a Banach space, and let $F\colon X\to Y$ be a lower semi-continuous multivalued map with non-empty, convex, and closed values.
Then $F$ has a continuous selection $f\colon X\to Y$.
Moreover, if $A$ is a closed subset of $X$ and $g\colon A\to Y$ is a continuous selection of $F|_A$ then $f$ can be assumed to extend $g$.
\end{theorem}

Recall that a multivalued map $F\colon X\to Y$ is \emph{lower semi-continuous} if for any open subset $S$ of $Y$, the set $\{x\in X\mid F(x)\cap S\ne \varnothing\}$ is open in $X$.

We now build up notation necessary for the proof of Theorem~\ref{thm:strong-local-contractibility}.
For any subset $U$ of $X$, we define
\begin{align*}
P_U &= \{ \mu\in \vietm{\cW}\mid \mathrm{supp}\, \mu \cap U\ne \varnothing\} \\
P^U &= \{ \mu\in \vietm{\cW} \mid \mathrm{supp}\, \mu \subseteq U\}.
\end{align*}
Furthermore, for a finite collection $\cU=\{U_1, U_2, \dots, U_n\}$ of subsets $U_i\subseteq X$, let
\[P_{\cU} = \bigcap_{i=1}^n P_{U_i}.\]

Propositions~\ref{prop:open},~\ref{prop:convex}, and~\ref{prop:convex2} have short proofs, which we omit.

\begin{proposition}
\label{prop:open}
If $U$ is an open subset of $X$, then $P_U$ is open in $\vietm{\cW}$.
If $\cU=\{U_1, \dots, U_n\}$ is a collection of open subsets of $X$, then $P_{\cU} $ is open in $\vietm{\cW}$.
\end{proposition}

\begin{proposition}
\label{prop:convex}
For any $Y\subseteq X$, the set $P^Y$ is a convex subset of $P(X)$.
In particular, $P^Y$ is contractible, and every $\mu\in P^Y$ is a strong deformation retract of $P^Y$.
\end{proposition}

Throughout this subsection we fix $\mu =\sum_{i=1}^n m_i \delta_{x_i}$ to be a measure in $\vietm{\cW}$.
Let $\cU=\{U_1, U_2, \dots, U_n\}$ be a collection of disjoint open subsets of $X$ with $x_i\in U_i$ for all $i$, and with $\cup\cU$ contained in an element of $\cW$.
For any $\varepsilon >0$, let
\[ P(\mu,\cU,\varepsilon) = \{ \nu \in \vietm{\cW}~:~\supp(\nu) \subseteq \cup\cU \mbox { and } | \nu(U_i) - \mu(U_i)|\le \varepsilon \mbox{ for all } i \}.\]

\begin{proposition}
\label{prop:convex2}
$P(\mu,\cU,\varepsilon)$ is convex.
\end{proposition}

Let $\tilde{B}(\mu)$ be an open ball in $P(X)$ with respect to the Wasserstein metric, centered at $\mu$, and let $B(\mu) = \tilde{B}(\mu)\cap \vietm{\cW}$.
Note that $\tilde{B}(\mu)$ is convex in $P(X)$, while $B(\mu)$ typically is not: the reason is that, for example, in the Vietoris--Rips metric thickenings, the union of two supports of diameter less than $r$ can have diameter above $r$.
Recall $\cU=\{U_1, U_2, \dots, U_n\}$ is a collection of disjoint open subsets of $X$ such that $x_i\in U_i$ for all $i$, and such that $\cup\cU$ is contained in an element of $\cW$.
We will prove that the open neighborhoods $P_{\cU}\cap B(\mu)$ of $\mu$ are strongly contractible for sufficiently small covers $\cU$.

Note that, for an infinite compactum $X$, the space $P(X)$  can be regarded as a convex compact infinite-dimensional subset of the normed space $\mathbb R^{C(X)}$.
Therefore it is affinely and topologically embeddable in the Hilbert (and hence Banach) space $l_2$ by Klee's theorem~\cite[Theorem~1.1]{klee1955some}.

%\textcolor{red}{Alex: I've added explanation why $P(X)$ can be viewed as a convex subset of $l_2$.}

Define a multivalued map $F\colon P_{\cU}\cap B(\mu) \to P(X)$ by
\begin{equation}
\label{eq:F}
F(\zeta)=\{\nu\in P(\mu,\cU,\varepsilon)~|~\supp(\nu)\subseteq\supp(\zeta)\}.
\end{equation}
 
We will use Michael's selection theorem~\cite{michael1956continuous} to obtain a continuous selection $f\colon P_{\cU}\cap B(\mu) \to P(\mu,\cU,\varepsilon)$.
Thus, we  need to prove that 
the values of $F$ are non-empty, convex, and closed, and that $F$ is lower semi-continuous.

\begin{proposition}
\label{prop:nonempty}
The values of $F$ are non-empty, convex, and closed for any $\zeta\in P_{\cU}\cap B(\mu)$.
\end{proposition}

\begin{proof}
By the assumption, $\supp(\zeta)$  contains a set of the form 
$\{ y^i_1,y^i_2,\dots, y^i_{n_i}\}_{i=1}^n$, where $y^i_j\in U_i$ for each $i$.
Let 
\[\nu =\sum_{i=1}^n \frac{m_i}{n_i} \sum_{j=1}^{n_i} \delta_{y^i_j},\]
where $\mu =\sum_{i=1}^n m_i \delta_{x_i}$.
It is easy to check that $\nu\in F(\zeta)$.

That $F$ is closed-valued follows from the fact that the support of all measures from $F(\zeta)$ is contained in the support of $\zeta$, and from the definition of $P(\mu,\cU,\varepsilon)$.
The convexity follows from Propositions~\ref{prop:convex} and~\ref{prop:convex2}.
\end{proof}

\begin{proposition} \label{prop:lsc}
The multivalued map $F\colon P_{\cU}\cap B(\mu) \to P(X)$ defined in~\eqref{eq:F} is lower semi-continuous.
\end{proposition}

\begin{proof}
To check that $F$ is lower semi-continuous, we must show that for any open subset $S$ of $P(X)$, the set $\{\xi\in P_{\cU}\cap B(\mu)\mid F(\xi)\cap S\ne \varnothing\}$ is open in $P_{\cU}\cap B(\mu)$.
We proceed along the following scheme: for any $\zeta \in P_{\cU}\cap B(\mu)$, any $\nu \in F(\zeta)$, and any arbitrarily small neighborhood $N$ of $\nu$, we define a small neighborhood $N'$ of $\zeta$ such that for each $\xi \in N'$, there exists $\lambda \in F(\xi) \cap N$.

Let
\[\zeta = \sum_{i=1}^n  \sum_{j=1}^{n_i} w^i_j \delta_{y^i_j}+\sum_{t=1}^k w_t \delta_{y_t},\]
where $y^i_j\in U_i$ for all $i,j$, and $y_t\notin \cup\cU$ for all $t=1,2,\dots,k$.
Pick $\nu\in F(\zeta)$.
Next, consider an open neighborhood of $\nu$ which, by Lemma~\ref{lemma:basis3}, may be assumed to be of the form $N=N(\nu,\widetilde \cU,\alpha)$.
Moreover, if 
$\cV=\{V^i_j\mid i=1,\dots,n,\, j=1,\dots, n_i\}$ is a collection of 
%$\{V^i_j\mid i=1,2,\dots ,n, j=1,2,\dots, n_i\}$ are 
disjoint open neighborhoods of the points $y^i_j$ such that $V^i_j\subseteq U_i$ for all $i$, 
and if $\cV'=\{V^i_j\mid y^i_j \in \supp \nu\}$, then
we may further assume that 
$N=N(\nu,\cV',\alpha)$.
%$N$ has the form $N(\nu, \{V^i_j\mid y^i_j \in \supp \nu\},\alpha)$.
Choose $\beta >0$ so that $\beta < w^i_j$ for all $i,j$, and so that
$N' \coloneqq N(\zeta,\cV,\beta)$
%\[N' \coloneqq N(\zeta, \{V^i_j\mid \ i=1,2,\dots,n, j=1,2,\dots ,n_i\},\beta)\]
is contained in $P_{\cU}\cap B(\mu)$, the domain of $F$.
Consider $\xi \in N'$.
Note that the choice of $\beta$ implies that $\supp \xi \cap V^i_j \ne \varnothing$ for all $i,j$.
For each $i=1,2,\dots ,n$ and $j=1,2,\dots, n_i$, pick $z^i_j\in \supp \xi \cap V^i_j$, and let 
\[\lambda = \sum_{
V^i_j\in \cV'
%\nu(V^i_j)\ne 0
} \nu (V^i_j) \delta_{z^i_j}.\]
It is easy to see that $\lambda (U_i) = \nu(U_i)$ for $i=1,2,\dots ,n$, and that $\supp\lambda \subseteq \supp \xi$.
Therefore $\lambda \in F(\xi)$.
Moreover, $\lambda (V^i_j) = \nu (V^i_j)$ for all $V^i_j\in \cV'$, and hence $\lambda \in N$.
Thus, $\lambda \in F(\xi)\cap N$, as required.
\end{proof}

\begin{theorem}
\label{thm:strong-local-contractibility}
If $X$ is a compact metrizable space then $\vietm{\cW}$ is strictly strongly locally contractible.
\end{theorem}

\begin{proof}
As everywhere in this section, let $\mu =\sum_{i=1}^n m_i \delta_{x_i}$ be a measure in $\vietm{\cW}$.
We must show that $\mu$ has a local base of neighborhoods which strongly deformation retract onto $\mu$.

Consider an open ball $B(\mu)$.
Because of Lemma~\ref{lemma:basis3}, we can find $\varepsilon >0$ and a collection $\cU = \{U_1, U_2, \dots, U_n\}$ of disjoint open neighborhoods of $x_1,x_2,\dots, x_n$, such that $\cup\cU$ is contained in an element of $\cW$, and $P(\mu,\cU,\varepsilon)\subseteq B(\mu)$.
By choosing sufficiently small $\varepsilon$ we can also require that $P(\mu,\cU,\varepsilon)\subseteq P_{\cU}$.

Let $F\colon P_{\cU}\cap B(\mu) \to P(X)$ be the multivalued map, defined, as before, by 
\[F(\zeta)=\{\nu\in P(\mu,\cU,\varepsilon)~|~\supp(\nu)\subseteq\supp(\zeta)\}.
\]
By Propositions~\ref{prop:nonempty} and~\ref{prop:lsc}, $F$ is a lower semi-continuous map with nonempty closed and convex values.
Therefore, we can apply Michael's selection theorem~\cite{michael1956continuous} (see Theorem~\ref{theorem:selection}) to obtain a continuous selection $f\colon P_{\cU}\cap B(\mu) \to P(\mu,\cU,\varepsilon)$.
Moreover, by the relative version of the selection theorem, we may assume that $f(\mu) = \mu$.

To construct a deformation retraction of $P_{\cU}\cap B(\mu)$ to $\mu$ that fixes $\mu$, we first join $\zeta$ to $f(\zeta)$ by a straight-line homotopy.
Indeed, note that for all $0\le t\le 1$ we have $(1-t)\zeta+t \cdot f(\zeta)\in \vietm{\cW}\cap \ \tilde{B}(\mu) = B(\mu)$ since $\supp f(\zeta)\subseteq \supp \zeta$, while the continuity of this straight-line homotopy follows from~\cite[Proposition~2.4]{AMMW} or~\cite[Lemma~3.9]{AAF}.
The image of $P_{\cU}\cap B(\mu)$ under this retraction is contained in $P(\mu,\cU,\varepsilon)$, which is convex by Proposition~\ref{prop:convex2}.
And any convex set containing $\mu$ can be strongly deformation retracted onto $\mu$, finishing the proof.
\end{proof}

\begin{corollary}
\label{cor:strong-local-contractibility}
Let $X$ be a compact metric space and let $r>0$.
Then $\vrm{X}{r}$ is strictly strongly locally contractible.
If $Z\supseteq X$ is a metric space extending the metric on $X$, then $\acechm{X}{Z}{r}$ is strictly strongly locally contractible.
\end{corollary}

\section{Conclusion}
\label{sec:conclusion}

We end with some open questions.

\begin{enumerate}

\item For $X$ compact but not necessarily finite-dimensional, is $\vrm{X}{r}$ an ANR or an ANE?

\item Is it true that for $X$ compact, any two sufficiently close maps into $\vrm{X}{r}$ are homotopic?
(This is known to be true for ANRs, and hence is a relaxation of Question (1).)

\item For $X$ an arbitrary metric space, is $\vrm{X}{r}$ strongly locally contractible?

\item When can a map $S^n\to \vrm{X}{r}$ from the $n$-sphere be deformed (up to homotopy) to have its image contained in $\vrm{X}{r} \cap P_{n+1}(X)$?

\item Let $X$ be an ANR.
When is $\vrm{X}{r} \cap P_{n+1}(X)$ an ANR?

%\item Latschev's theorem~\cite{Latschev2001} states that for any $r>0$ sufficiently small, there is a $\delta>0$ such that if $Y$ is $\delta$-close to $M$ in the Gromov--Hausdorff distance, then $\vr{Y}{r'}$ is homotopy equivalent to $M$ for all $r'<r$.
%This is also true for the metric thickening $\vrm{Y}{r'}$ when $Y$ is finite, simply because for $Y$ finite the Vietoris--Rips simplicial complex and metric thickening are homeomorphic.
%Is there a version of Latschev's theorem with metric thickenings that allows for infinite $Y$?
%\note{We can now answer affirmitively, using Corollary~\ref{cor:main}, when $Y$ is Polish and finite-dimensional!
%Could still keep a question about a setting with fewer assumptions.}

\end{enumerate}

\bibliographystyle{plain}
\bibliography{VietorisThickeningsAndComplexesOfManifoldsAreHomotopyEquivalent}

\end{document}